\documentclass{amsart}
\usepackage{verbatim, rotating, stmaryrd, bbm, pict2e, ytableau, mathrsfs, amssymb}
\usepackage[colorlinks]{hyperref}
\usepackage[all]{xy}
\usepackage{cleveref}

\usepackage{tikz}
\usetikzlibrary{matrix,arrows,backgrounds}

\makeatletter
\def\eqref{\@ifstar\@eqref\@@eqref}
\def\@eqref#1{\textup{\tagform@{\ref*{#1}}}}
\def\@@eqref#1{\textup{\tagform@{\ref{#1}}}}
\newcommand\@dotsep{4.5}
\def\@tocline#1#2#3#4#5#6#7{\relax
  \ifnum #1>\c@tocdepth
  \else
    \par \addpenalty\@secpenalty\addvspace{#2}%
    \begingroup \hyphenpenalty\@M
    \@ifempty{#4}{%
      \@tempdima\csname r@tocindent\number#1\endcsname\relax
    }{%
      \@tempdima#4\relax
    }%
    \parindent\z@ \leftskip#3\relax \advance\leftskip\@tempdima\relax
    \rightskip\@pnumwidth plus1em \parfillskip-\@pnumwidth
    #5\leavevmode\hskip-\@tempdima #6\relax
    \leaders\hbox{$\m@th
      \mkern \@dotsep mu\hbox{.}\mkern \@dotsep mu$}\hfill
    \hbox to\@pnumwidth{\@tocpagenum{#7}}\par
    \nobreak
    \endgroup
  \fi}
\def\l@subsection{\@tocline{2}{0pt}{20pt}{5pc}{}}
\def\l@subsubsection{\@tocline{2}{0pt}{30pt}{5pc}{}}
\makeatother

\ytableausetup{boxsize=11.5pt}
\ytableausetup{aligntableaux=center}
\definecolor{lred}{RGB}{220,60,60}

\theoremstyle{plain}
\newtheorem{Thm}{Theorem}[section]
\newtheorem*{Thm*}{Theorem}
\newtheorem{Cor}[Thm]{Corollary}
\newtheorem*{CorA}{\Cref{corSupp}}
\newtheorem*{CorB}{\Cref{corRegProj}}

\newtheorem{Lem}[Thm]{Lemma}
\newtheorem{Prop}[Thm]{Proposition}

\theoremstyle{definition}
\newtheorem{Def}[Thm]{Definition}
\newtheorem{Ex}[Thm]{Example}
\newtheorem{Ques}{Question}

\theoremstyle{remark}
\newtheorem*{Rmk}{Remark}

\DeclareMathOperator{\im}{im}
\DeclareMathOperator{\spec}{Spec}
\DeclareMathOperator{\proj}{Proj}
\DeclareMathOperator{\coker}{coker}
\DeclareMathOperator{\SL}{SL}
\DeclareMathOperator{\GL}{GL}
\DeclareMathOperator{\Sp}{Sp}
\DeclareMathOperator{\B}{B}
\DeclareMathOperator{\U}{U}
\DeclareMathOperator{\Lie}{Lie}

\DeclareMathOperator{\reg}{Reg}
\DeclareMathOperator{\supp}{Supp}
\DeclareMathOperator{\jtype}{JType}
\DeclareMathOperator{\rank}{rank}

\DeclareMathOperator{\Inj}{Inj}
\renewcommand{\hom}{\mathrm{Hom}}
\newcommand{\red}{\mathrm{red}}
\newcommand{\slt}[1][2]{\mathfrak{sl}_{#1}}

\newcommand{\set}[1]{\left\{#1\right\}}
\newcommand{\id}[1][\mbox{}]{\mathrm{id}_{#1}}

\newcommand{\catd}[2]{{\tt #2}(#1)}

\newcommand{\gim}[2][\mbox{}]{\im\Theta^{#1}_{#2}}
\newcommand{\gker}[2][\mbox{}]{\ker\Theta^{#1}_{#2}}
\newcommand{\gcoker}[2][\mbox{}]{\coker\Theta^{#1}_{#2}}
\newcommand{\Ho}[1][1]{\mathcal H^{[#1]}}
\newcommand{\F}{\mathscr F}
\newcommand{\K}[1][1]{\ker(\Ho)}
\newcommand{\PG}[1][G]{\mathbb P(#1)}
\newcommand{\OPG}[1][\mbox{}]{\mathcal O_{\mathbb P(G)#1}}
\newcommand{\coh}{\tt Coh}
\newcommand{\cohPG}{\mathtt{Coh}{(\mathbb P(G))}}
\newcommand{\D}[1]{\mathsf D(#1)}
\newcommand{\Db}[1]{\mathsf D^b(#1)}
\newcommand{\thick}[2][\mbox{}]{\mathsf{thick}_{#1}(#2)}
\newcommand{\KInj}[1]{\mathsf{K}(\Inj #1)}
\newcommand{\kalg}[1][k]{#1\text{-\tt alg}}
\newcommand{\kgrp}[1][k]{#1\text{-\tt grp}}

\begin{document}

\title[Detecting projectivity in sheaves associated to representations]{Detecting projectivity in sheaves associated to representations of infinitesimal groups}

\author{Jim Stark}

\thanks{The author was partially supported by the NSF grant DMS-0953011}

\address{Department of Mathematics\\
   University of Washington\\
   Seattle, WA 98105}

\email{jstarx@gmail.com}

\date{5 September, 2013}

\begin{abstract}
Let $G$ be an infinitesimal group scheme of finite height $r$ and $V(G)$ the scheme which represents $1$-parameter subgroups of $G$.  We consider sheaves over the projective support variety $\PG$ constructed from a $G$-module $M$.  We show that if $\PG$ is regular then the sheaf $\Ho(M)$ is zero if and only if $M$ is projective.  In general, $\Ho$ defines a functor from the stable module category and we prove that its kernel is a thick triangulated subcategory.  Finally, we give examples of $G$ such that $\PG$ is regular and indicate, in characteristic $2$, the connection to the BGG correspondence.  Along the way we will provide new proofs of some known results and correct some errors in the literature.
\end{abstract}

\maketitle

\tableofcontents

\setcounter{section}{-1}
\section{Introduction}

The theory of support varieties/schemes for an infinitesimal group scheme $G$ was established by Suslin, Friedlander, and Bendel~\cite{bfs1paramCoh, bfsSupportVarieties} in 1997.  It unifies work done by Carlson~\cite{carlsonCohVar} and Friedlander and Parshall~\cite{friedparshallSupport} in the 80's on the spectrum of the cohomology of, respectively, an elementary abelian group and a restricted Lie algebra.  The projective version, $\PG$, of these support schemes provides a useful geometric setting on which the more general notion of local Jordan type can be defined.  This includes, in particular, the definition of modules of constant Jordan type which have received considerable attention in recent years.  Friedlander and Pevtsova~\cite{friedpevConstructions} and Benson and Pevtsova~\cite{benPevtVectorBundles} have introduced various geometric constructions which, given a representation $M$ of $G$, yield coherent sheaves over $\PG$.  Information about the local Jordan type of $M$ is reflected in the properties of these sheaves, most notably that modules of constant Jordan type are those modules for which certain constructions produce exclusively vector bundles.  Moreover, one can even test if $M$ belongs to the often studied class of endotrivial modules, or whether $M$ has the equal images or equal kernels property introduced by Carlson, Friedlander, and Pevtsova~\cite{cfsZpZpModules}, by testing whether certain sheaves produced from $M$ are locally free.  In the current paper we investigate two open questions concerning $G$-modules and their associated sheaves: How the projectivity of a module $M$ is reflected in a particular sheaf, $\Ho(M)$, and whether the support of $\Ho(M)$ can be given a representation theoretic description.

We begin in \cref{secRev} by summarizing the definition and relevant results concerning support schemes and their projective version $\PG$.  We will construct the global operator $\Theta_M$ and give several relevant examples.  In \cref{secJType} we define and discuss the local Jordan type of $M$ in terms of partitions and Young diagrams.  Most notably we define the support $\PG_M \subseteq \PG$ as the set of points at which the global operator $\Theta_M$ does not yield a projective $K[t]/t^p$-module.

In \cref{secAssSh} we define the kernel, image, and cokernel sheaves associated to the global operator $\Theta_M$ and we define the sheaves
\[\F_i(M) = \frac{\gker{M} \cap \gim[i-1]{M}}{\gker{M} \cap \gim[i]{M}}\]
for $1 \leq i \leq p$.  The $\F_i(M)$ were first defined for the case of an elementary abelian group of rank $r$~\cite{benPevtVectorBundles}.  There we have $\PG = \mathbb P^{r-1}$ and Benson and Pevtsova prove a realization theorem for vector bundles over this projective space; specifically, if $\mathscr G$ is a vector bundle of rank $s$ then, up to a Frobenius twist, $\mathscr G$ can be realized as $\F_1(M)$ for some module $M$ of constant Jordan type $[p]^n[1]^s$.  When $p = 2$ the Frobenius twist is not present and Benson and Pevtsova note that their result can be thought of as a version of the Bern\u{s}te\u{\i}n, Gel'fand, Gel'fand (BGG) correspondence between modules for an exterior algebra and vector bundles over $\mathbb P^{r-1}$ (we discuss this connection in \cref{secDG}).

The $\F_i(M)$ reflect many properties of the module $M$; in particular, we give a local version of the result by Benson and Pevtsova that a module $M$ has constant Jordan type if and only if each $\F_i(M)$ is locally free and, moreover, that the ranks of the $\F_i(M)$ yield the Jordan type of $M$.  In addition to this result, the $\F_i(M)$ are significant in that they are isomorphic to the quotients of filtrations of many of the sheaves that we are interested in.

The main theoretical results are contained in \cref{secHo}.  Here we introduce the sheaves
\[\Ho[i](M) = \frac{\gker[i]{M}}{\gim[p-i]{M}}\]
and ask the following two questions:

\begin{Ques}
Is $M$ projective if and only if $\Ho(M) = 0$?
\end{Ques}

\begin{Ques}
Given a $G$-module $M$ do we have $\PG_M = \supp\Ho(M)$?
\end{Ques}

The sheaves $\Ho[i](M)$ were first defined by Friedlander and Pevtsova using the notation $\mathcal M^{[i]}$~\cite[5.14]{friedpevConstructions}.  They were motivated by work of Duflo and Serganova~\cite{dufloSerganova} on Lie superalgebras.  Serganova has indicated in private correspondence that \cref*{quesMain} has an affirmative answer in the case of a Lie superalgebra $\mathfrak g_0 \oplus \mathfrak g_1$ which satisfies the additional assumption that the self commuting elements of $\mathfrak g_1$ span $\mathfrak g_1$.  This is a mild assumption satisfied, in particular, by any simple classical Lie superalgebra except $\mathfrak{osp}(1|2n)$.  Duflo and Serganova note that in case $\mathfrak g = \mathfrak{osp}(1|2n)$ every finite-dimensional $\mathfrak g$-module is projective anyway~\cite[3.6]{dufloSerganova}.

In \cref{thmHoProj} we obtain a sizable list of conditions, any one of which implies that $M$ is projective under the assumption $\Ho(M) = 0$.  We show that if $\reg\PG$ is the regular locus of $\PG$ then $\PG_M \cap \reg\PG \subseteq \supp\Ho(M)$.  This answers \cref*{quesSupp}, and consequently \cref*{quesMain}, in the affirmative at least in the case that $\PG$ is smooth:

\begin{CorA}
If $\PG$ is regular then $\PG_M = \supp\Ho(M)$.
\end{CorA}

\begin{CorB}
If $\PG$ is regular then $M$ is projective if and only if $\Ho(M) = 0$.
\end{CorB}

For the remainder of \cref{secHo} we consider the full subcategory $\K$ of the stable module category $\catd{G}{\underline{mod}}$ consisting of those objects $M$ which satisfy $\Ho(M) = 0$.  We show that this is a thick triangulated subcategory of $\catd{G}{\underline{mod}}$ and we give an example where $\K$ has infinite representation type, thus providing, in general, an answer to Questions \ref{quesMain} and \ref{quesSupp} in the negative.

In \cref{secReg} we give additional explicit computations of $\PG$ for various $G$.  In light of the results of \cref{secHo} we focus on examples of $G$ such that $\PG$ is regular.  We end, in \cref{secDG}, by reminding the reader of the BGG correspondence.  We give a generalization to DG-modules over a polynomial ring which is a slight modification of the correspondence obtained by Benson et al.~\cite[5.5]{bikLocCohSup} (which was in turn inspired by a correspondence obtained by Avramov et al.~\cite[7]{abimHomPerf}) and we show that $\Ho$ factors through this correspondence.  We use this to obtain further results on the functor $\Ho$ in the case of representations of an elementary abelian group.

\section{The Global Operator of an Infinitesimal Group Scheme} \label{secRev}

In this section we summarize the definitions required to construct the associated sheaves of \cref{secAssSh}.  For a complete review we refer the reader to Suslin et al.~\cite{bfsSupportVarieties} and to Friedlander and Pevtsova~\cite{friedpevConstructions}.  We fix, once and for all, an algebraically closed field $k$ of positive characteristic $p$.

Recall that a group scheme $G$ (over $k$) is \emph{infinitesimal} of height at most $r$ if the coordinate algebra $k[G]$ is a finite dimensional local ring and $x^{p^r} = 0$ for all $x$ contained in the maximal ideal.  To give $G$ it suffices to designate a commutative Hopf algebra as the coordinate ring $k[G]$ or designate a cocommutative Hopf algebra as its linear dual: the group ring $kG = \hom_k(k[G], k)$.  The group ring is significant in particular because a representation of $G$ is equivalent to a $kG$-module.  Our two main examples, \ref{exLie1} and \ref{exGr}, will be height $1$ group schemes whose representation theory is equivalent to the representation theory of a restricted Lie algebra and an elementary abelian group, respectively.

\begin{Ex} \label{exLie1}
Let $\mathfrak g$ be a restricted Lie algebra (always assumed to be finite dimensional) and $\mathcal U_p(\mathfrak g)$ its restricted universal enveloping algebra.  This is a cocommutative Hopf algebra whose primitive elements are exactly the elements of the Lie algebra $\mathfrak g \subseteq \mathcal U_p(\mathfrak g)$.  We define a group scheme $\underline{\mathfrak g}$ by designating the group ring $k\underline{\mathfrak g} = \mathcal U_p(\mathfrak g)$; it is an infinitesimal height $1$ scheme~\cite[I.8.5.b]{jantzen}.  Now $\underline{\mathfrak g}$-modules are equivalent to $\mathcal U_p(\mathfrak g)$-modules, i.e., representations of $\mathfrak g$ as a restricted Lie algebra.
\end{Ex}

\begin{Ex} \label{exFrob}
Let $G$ be an algebraic group and $r \in \mathbb N$.  The \emph{$r^\text{th}$ Frobenius kernel} of $G$ is denoted $G_{(r)}$.  It is the group scheme with coordinate ring
\[k[G_{(r)}] = \frac{k[G]}{\langle f^{p^r} \mid f \in I_1\rangle}\]
where $I_1 \subseteq k[G]$ is the augmentation ideal.  It is infinitesimal of height at most $r$.  Note that if $\mathfrak g = \Lie(G)$ then the schemes $\underline{\mathfrak g}$ and $G_{(1)}$ can be identified; see Jantzen~\cite[I.9]{jantzen} for details.
\end{Ex}

\begin{Ex} \label{exGr2}
The $r^\text{th}$ Frobenius kernel, $\mathbb G_{a(r)}$, of the additive group is an infinitesimal height $r$ group scheme.  The coordinate ring is $k[\mathbb G_{a(r)}] = k[t]/t^{p^r}$ with Hopf structure given by designating $t$ a primitive element.  As a $k$-algebra the group ring of $\mathbb G_{a(r)}$ is
\[k\mathbb G_{a(r)} = \frac{k[u_1, \ldots, u_r]}{u_1^p, \ldots, u_r^p}\]
where $u_i$ is dual to $t^{p^{i-1}}$ in the monomial basis.  This is isomorphic to the group ring of the elementary abelian group $E = (\mathbb Z/p)^r$ but with a different Hopf structure.  Consequently, the equivalence between modules over $\mathbb G_{a(r)}$ and modules over $E$ does not extend to their tensor monoidal structures.
\end{Ex}

\begin{Ex} \label{exGr}
The group scheme $\mathbb G_{a(1)}^r$ is infinitesimal of height $1$.  The coordinate ring is $k[\mathbb G_{a(1)}^r] = k[t_1, \ldots, t_r]/(t_1^p, \ldots, t_r^p)$ with Hopf structure given by designating each $t_i$ a primitive element.  As a $k$-algebra the group ring of $\mathbb G_{a(1)}^r$ is
\[k\mathbb G_{a(1)}^r = \frac{k[u_1, \ldots, u_r]}{u_1^p, \ldots, u_r^p}\]
where $u_i$ is dual to $t_i$ in the monomial basis.  Once again, this group ring is isomorphic as a $k$-algebra to the group ring $k(\mathbb Z/p)^r$ but with a different Hopf structure: Each $u_i$ is a primitive element.  As with the previous example, the equivalence between $\mathbb G_{a(1)}^r$-modules and $(\mathbb Z/p)^r$-modules does not extend to their tensor monoidal structures.
\end{Ex}

\begin{Rmk}
The standard Hopf structure on the group ring of $E = (\mathbb Z/p)^r$ yields a group scheme which is finite but not infinitesimal.  Our methods apply only to infinitesimal schemes so we will use $\mathbb G_{a(1)}^r$ when considering modules for an elementary abelian group of rank $r$.
\end{Rmk}

A special role is played by the $\mathbb G_{a(r)}$ from \cref{exGr2}.  Denote by $G_R$ the base extension of $G$ to a commutative $k$-algebra $R$; this is a group scheme over $R$ with coordinate ring $R[G_R] = R \otimes_k k[G]$ and group ring $RG_R = R \otimes_k kG$.  A \emph{$1$-parameter subgroup} of height $r$ is a homomorphism $\mathbb G_{a(r), R} \to G_R$ of group schemes over some $R$.  The collection of all such homomorphisms defines the functor of points of the \emph{support scheme} $V(G)$ of $G$.

\begin{Thm}[{\cite[1.5, 1.12]{bfs1paramCoh}}] \label{thmVG}
Let $G$ be an infinitesimal group scheme of height $r$.  Then there exists an affine scheme $V(G) = \spec k[V(G)]$ whose functor of points
\[V(G)(-) = \hom_{\kalg}(k[V(G)], -)\]
is naturally isomorphic to the functor
\[R \mapsto \hom_{\kgrp[R]}(\mathbb G_{a(r), R}, G_R)\]
from commutative $k$-algebras to sets.  Moreover, $k[V(G)]$ is a finitely generated connected graded $k$-algebra with homogeneous generators of degree $p^i$ for $0 \leq i < r$.
\end{Thm}

\begin{Rmk}
For $G$ infinitesimal of height $1$ the coordinate ring $k[V(G)]$ is generated in degree $1$, but for $G$ of larger heights this need not be the case.
\end{Rmk}

\begin{Ex} \label{exVlie}
Let $\mathfrak g$ be a restricted Lie algebra with basis $\set{g_1, \ldots, g_n}$ and dual basis $\set{x_1, \ldots, x_n}$.  Given a commutative $k$-algebra $R$ we extend the $p$-operation to $(-)^{[p]}\colon R \otimes_k \mathfrak g \to R \otimes_k \mathfrak g$ via $a \otimes v \mapsto a^p \otimes v^{[p]}$.  Choose $f_1, \ldots, f_n \in k[x_1, \ldots, x_n]$ such that
\[(x_1 \otimes g_1 + \cdots + x_n \otimes g_n)^{[p]} = f_1 \otimes g_1 + \cdots + f_n \otimes g_n.\]
We define the \emph{restricted nullcone} of $\mathfrak g$ to be $\mathcal N_p(\mathfrak g) = \spec k[x_1, \ldots, x_n]/(f_1, \ldots, f_n)$.  This is the scheme whose functor of points is given by
\[\mathcal N_p(\mathfrak g)(R) = \set{v \in R \otimes_k \mathfrak g \ \middle| \ v^{[p]} = 0}.\]
Note that the $k$-points $\mathcal N_p(\mathfrak g)(k) \subseteq \mathfrak g$ of this scheme give the traditional definition of the restricted nullcone, but the scheme $\mathcal N_p(\mathfrak g)$ need not be reduced.

The group $\underline{\mathfrak g}$ from \cref{exLie1} has support scheme $\mathcal N_p(\mathfrak g)$.  The isomorphism $\mathcal N_p(\mathfrak g)(R) \simeq \hom_{\kgrp[R]}(\mathbb G_{a(1), R}, \underline{\mathfrak g}_R)$ sends a $p$-nilpotent $v \in R \otimes_k \mathfrak g$ to the homomorphism $\mathbb G_{a(1), R} \to \underline{\mathfrak g}_R$ whose induced map on group rings, $R[x]/x^p \to R \otimes_k \mathcal U_p(\mathfrak g)$, is defined by $x \mapsto v$.
\end{Ex}

\begin{Def}
Let $G$ be an algebraic group and $\phi\colon G \to \GL_n$ a closed embedding.  If, for each $p$-nilpotent $x \in \Lie(G)$, the exponential map $t \mapsto \exp(\mathrm{d}\phi(tx))$ takes values in $G$ then we say that $\phi$ is an embedding of \emph{exponential type}.  If $G$ has such an embedding then the group $G$ is of exponential type.
\end{Def}

By inspection the following groups are of exponential type: $\GL_n$, $\SL_n$, $\Sp_{2n}$, $\B_n$ (upper triangular $n \times n$ matrices), $\U_n$ (strictly upper triangular $n \times n$ matrices), and the orthogonal group $O(\phi)$ associated to a non-degenerate bilinear form $\phi$ (see Suslin et al.~\cite[1.8]{bfs1paramCoh}, McNinch~\cite{mcninchExpType}, and Sobaje~\cite{sobajeExpOneParam} for further discussion of the exponential type condition).

\begin{Ex} \label{exVlier}
Let $G$ be of exponential type, $\mathfrak g = \Lie(G)$, and $r \in \mathbb N$.  Define $\mathcal N_p^{[r]}(\mathfrak g)$ to be the closed subscheme of $(\mathcal N_p(\mathfrak g))^r$ consisting of $r$-tuples whose elements pairwise commute; then the support scheme of $G_{(r)}$ is $V(G_{(r)}) = \mathcal N_p^{[r]}(\mathfrak g)$.  Given $(\alpha_1, \ldots, \alpha_r) \in \mathcal N_p^{[r]}(\mathfrak g)$ the coordinate functions for the elements of the matrix $\alpha_i$ have degree $p^{i-1}$.  See Suslin et al.~\cite[1.7]{bfs1paramCoh} and Friedlander and Pevtsova~\cite[2.10]{friedpevConstructions} for more details.
\end{Ex}

\begin{Ex} \label{exVE}
The support scheme of the group scheme $\mathbb G_{a(1)}^r$ from \cref{exGr} is $\mathbb A^r = \spec k[x_1, \ldots, x_r]$.  The isomorphism $\mathbb A^r(R) \simeq \hom_{\kgrp[R]}(\mathbb G_{a(1), R}, \mathbb G_{a(1), R}^r)$ sends $(a_1, \ldots, a_r) \in R^r$ to the homomorphism $\mathbb G_{a(1), R} \to \mathbb G_{a(1), R}^r$ whose induced map on coordinate rings, $R[t]/t^p \to R[t_1, \ldots, t_r]/(t_1^p, \ldots, t_r^p)$, is defined by $t \mapsto \sum_ia_it_i$.
\end{Ex}

\begin{Ex} \label{exVGr}
The support scheme of the group scheme $\mathbb G_{a(r)}$ from \cref{exGr2} is $\mathbb A^r = \spec k[x_1, \ldots, x_r]$.  The isomorphism $\mathbb A^r(R) \simeq \hom_{\kgrp[R]}(\mathbb G_{a(r), R}, \mathbb G_{a(r), R})$ sends an element $(a_1, \ldots, a_r) \in R^r$ to the homomorphism $\mathbb G_{a(r), R} \to \mathbb G_{a(r), R}$ whose induced map on coordinate rings, $R[t]/t^{p^r} \to R[t]/t^{p^r}$, is defined by $t \mapsto \sum_ia_it^{p^{i-1}}$.
\end{Ex}

The grading of $k[V(G)]$ is given as follows.  From \cref{exVGr} we have $V(\mathbb G_{a(r)}) = \mathbb A^r$ so we get a right monoid action $V(G) \times \mathbb A^r \to V(G)$ via composition of $1$-parameter subgroups.  Restrict to an action $V(G) \times \mathbb A^1 \to V(G)$ by including $\mathbb A^1 \subset \mathbb A^r$ as the first factor (these are the $1$-parameter subgroups of $\mathbb G_{a(r)}$ whose maps on coordinate rings $k[t]/t^{p^r} \to k[t]/t^{p^r}$ are degree preserving).  If $\phi\colon k[V(G)] \to k[V(G)] \otimes_k k[t]$ is the associated comorphism then the algebra $k[V(G)]$ is graded by $k[V(G)]_n = \phi^{-1}(k[V(G)] \otimes_k t^n)$~\cite[1.11]{bfs1paramCoh}.

\begin{Ex} \label{exVGrt}
To determine the grading on $k[V(\mathbb G_{a(r)})] = k[x_1, \ldots, x_r]$ from \cref{exVGr} note that the monoid action $\mathbb A^r \times \mathbb A^1 \to \mathbb A^r$ is given by
\[(a_1, a_2, \ldots, a_r)\cdot b = (a_1b, a_2b^p, \ldots, a_rb^{p^{r-1}}).\]
Its comorphism $k[x_1, \ldots, x_r] \to k[x_1, \ldots, x_r] \otimes_k k[t]$ is defined by $x_i \mapsto x_i \otimes t^{p^{i-1}}$ therefore $\deg x_i = p^{i-1}$.
\end{Ex}

\begin{Def}
Let $\PG = \proj k[V(G)]/\mathfrak N$ where $\mathfrak N$ is the nilradical of $k[V(G)]$.
\end{Def}

\begin{Rmk}
Reducing $V(G)$ before taking $\proj$ isn't strictly necessary.  Most of the results that follow work over $\proj k[V(G)]$ as well, but for some directions in \cref{thm413} and \cref{thmHoProj} we must assume that the open subscheme $U \subseteq \PG$ is reduced.  As $V(G)$ is not, in general, reduced there are more applications of the theory if we reduce $V(G)$ here.
\end{Rmk}

The scheme $\PG$ is the geometric space over which we will construct sheaves in \cref{secAssSh}.  We will construct these sheaves from kernels, images, and cokernels of certain global operators, $\Theta_M$, which we now describe.

The generic element of the functor of points of $V(G)$ is the homomorphism of group schemes
\[\mathcal U_G\colon\mathbb G_{a(r), k[V(G)]} \to G_{k[V(G)]}\]
in $V(G)(k[V(G)]) \simeq \hom_{\kalg}(k[V(G)], k[V(G)])$ corresponding to the identity map.  This induces a Hopf $k[V(G)]$-algebra homomorphism between the corresponding group rings
\[\mathcal U_{G, \ast}\colon k[V(G)] \otimes_k \frac{k[u_1, \ldots, u_r]}{(u_1^p, \ldots, u_r^p)} \to k[V(G)] \otimes_k kG.\]
Given a $kG$-module $M$, its extension of scalars $k[V(G)] \otimes_k M$ is a module over $k[V(G)] \otimes_k kG$ and multiplication by $\Theta_G = \mathcal U_{G, \ast}(1 \otimes u_r)$ gives a homomorphism
\[\Theta_M\colon k[V(G)] \otimes_k M \to k[V(G)] \otimes_k M\]
of $k[V(G)]$-modules.  This map is homogeneous of degree $p^{r-1}$~\cite[2.11]{friedpevConstructions}; hence, after reducing, it induces a map of sheaves over $\PG$ which, by some abuse of notation, we will also denote $\Theta_M$.

\begin{Def}
Given a $G$-module $M$ we define $\widetilde M = \OPG \otimes_k M$.  The \emph{global operator} corresponding to $M$ is the sheaf map
\[\Theta_M\colon \widetilde M \to \widetilde M(p^{r-1})\]
induced by the action of $\Theta_G = \mathcal U_{G, \ast}(1 \otimes u_r)$.
\end{Def}

\begin{Rmk}
We caution the reader that $\OPG(1)$ need not be locally free when the ring $k[V(G)]$ is not generated in degree $1$.  Instead one has that for any integer $d \in \mathbb Z$ the sheaf $\OPG(dp^{r-1})$ is locally free of rank $1$.  This follows from the fact that the generators of $k[V(G)]$ have order dividing $p^{r-1}$ (c.f.\ \cref{thmVG} and Friedlander and Pevtsova~\cite[4.5]{friedpevConstructions}).  Consequently, $\widetilde M(dp^{r-1})$ is locally free of rank $\dim M$ and its specialization at a point $v \in \PG$ is $k(v) \otimes_{\OPG[,v]} \widetilde M(dp^{r-1}) \simeq k(v) \otimes_k M$.
\end{Rmk}

Observe that $M \mapsto \widetilde M$ gives an exact functor from the category, $\catd{G}{mod}$, of finitely generated $G$-modules to the category, $\cohPG$, of coherent sheaves over $\PG$.  Global operators are natural with respect to this functor, i.e., given a module homomorphism $\phi\colon M \to N$ the diagram
\[\xymatrix{\widetilde M \ar[r]^-{\Theta_M} \ar[d]_-{\id \otimes \phi} & \widetilde M(p^{r-1}) \ar[d]^-{\id \otimes \phi} \\ \widetilde N \ar[r]_-{\Theta_N} & \widetilde N(p^{r-1})}\]
commutes.

\begin{Ex}
When $G = \underline{\mathfrak g}$ we have $\PG = \proj k[x_1, \ldots, x_n]/\sqrt{f_1, \ldots, f_n}$ from \cref{exVlie}.  This is the reduced scheme corresponding to the projective variety given by the traditional restricted nullcone $\mathcal N_p(\mathfrak g)(k) \subseteq \mathfrak g$.  The global operator $\Theta_M$ is given by the action of $\Theta_{\underline{\mathfrak g}} = x_1 \otimes g_1 + \cdots + x_n \otimes g_n$.
\end{Ex}

\begin{Ex}
When $G = \mathbb G_{a(r)}$ \cref{exVGrt} gives that $\PG$ is the weighted projective space $\PG[1, p, \ldots, p^{r-1}]$, i.e., we have $\PG = \proj k[x_1, \ldots, x_r]$ but with $\deg x_i = p^{i-1}$.  Given a module $M$ the global operator $\Theta_M$ is given by the action of $\Theta_{\mathbb G_{a(r)}} = x_1^{p^{r-1}} \otimes u_1 + x_2^{p^{r-2}} \otimes u_2 + \cdots + x_r \otimes u_r$.  Note that the grading $\deg x_i = p^{i-1}$ implies $\deg x_i^{p^{r-i}} = p^{r-1}$ as required.
\end{Ex}

\begin{Ex}
When $G = \mathbb G_{a(1)}^r$ we have $\PG = \mathbb P^{r-1}$ from \cref{exVE}.  Given a module $M$ the map $\Theta_M$ is given by the action of $\Theta_{\mathbb G^r_{a(1)}} = x_1 \otimes u_1 + \cdots + x_n \otimes u_n$.  Note that $\Theta_{\mathbb G_{a(1)}^r}$ is the operator studied by Benson and Pevtsova~\cite{benPevtVectorBundles}.
\end{Ex}

\section{Local Jordan Type} \label{secJType}

We now define the local Jordan type of a module by using the global operator $\Theta_M$ to associate a $p$-restricted partition to each point $v \in \PG$ in the underlying topological space of $\PG$, i.e., to each homogeneous prime ideal in $k[V(G)]$.

\begin{Def}
A \emph{$p$-restricted partition} is a finite sequence of non-negative integers $\lambda = (\lambda_1, \lambda_2, \ldots, \lambda_n)$ satisfying $p \geq \lambda_1 \geq \lambda_2 \geq \cdots \geq \lambda_n$.  We call the $\lambda_i$ the \emph{parts} of $\lambda$ and will also write such partitions in exponential notation where exponents denote repeated parts, for example $(5, 2, 2, 2, 1) = [5][2]^3[1]$.
\end{Def}

Consider a finite dimensional vector space $V$.  Given a linear operator $T\colon V \to V$ satisfying $T^p = 0$ we get a $k[t]/t^p$-module structure on $V$ by letting $t$ act via $T$.  Each indecomposable $k[t]/t^p$-module is isomorphic to one of $k[t]/t^i$ where $1 \leq i \leq p$.  If
\[V \simeq \bigoplus_{1 \leq i \leq p}\left(k[t]/t^i\right)^{a_i}\]
then we say the \emph{Jordan type} of the operator $T$ is $\jtype(T) = [p]^{a_p}[p-1]^{a_{p-1}}\cdots[1]^{a_1}$.  Note that $a_i$ is the number of blocks of size $i$ in the Jordan normal form of any matrix representation of $T$.

\begin{Def}
Let $M$ be a $G$-module.  The \emph{local operator} at a point $v \in \PG$ is the linear operator $\theta_{v, M} = k(v) \otimes_{\OPG[,v]} \Theta_M$ on $k(v) \otimes_k M$.  Let $\mathscr P_p$ be the set of all $p$-restricted partitions.  The \emph{local Jordan type} of $M$ is the function
\[\jtype(-, M)\colon\PG \to \mathscr P_p\]
defined by $\jtype(v, M) = \jtype(\theta_{v, M})$.
\end{Def}

\begin{Rmk}
This definition is slightly different from the one given by Friedlander and Pevtsova~\cite{friedpevConstructions}.  They define the local operator as the specialization of $\Theta_M$ when considered as a map of sheaves over $V(G)$ and thus associate to each point in $V(G)$ a partition.  One can check that the function $V(G) \to \mathscr P_p$ defined in this way depends only on the homogeneous primes in $V(G)$ and at these homogeneous primes agrees with the Jordan type $\PG \to \mathscr P_p$ as we have defined it above.  See Stark~\cite[Section 3.2]{starkGeneral} for a detailed proof of this in the case of a height $1$ scheme.  The proof goes through in the general case with only minor modifications.  Also see Friedlander and Pevtsova~\cite[Section 4]{friedpevConstructions} for the connection to the $\pi$-points definition of Jordan type used by Carlson et al.~\cite{cfpConstJType}, Friedlander et al.~\cite{fpsGenMaxJType}, and Friedlander and Pevtsova~\cite{friedpevPiSupp}.
\end{Rmk}

The Young diagram of a partition will be a useful visualization tool.  A \emph{Young diagram} is a two dimensional array of finitely many boxes whose row lengths are weakly decreasing.  Young diagrams correspond to partitions by reading row lengths from top to bottom and henceforth we will identify these two objects.  We also remind the reader that the \emph{conjugate} of a partition is the partition obtained by transposing the Young diagram.

\begin{Ex}
The partitions $[3][2]^2$, $[3][2][1]$, and $[3]^2[1]$ are given by
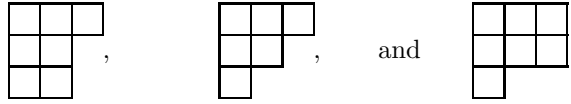
\[\ydiagram{3,2,2}, \qquad\qquad \ydiagram{3,2,1}, \qquad\text{and}\qquad \ydiagram{3,3,1}\]
respectively and $[3][2]^2$ and $[3]^2[1]$ are conjugate.
\end{Ex}

The $k[t]/t^p$-module $k[t]/t^i$ corresponds to the partition $[i]$ given by $1$ row of length $i$.
\[\ydiagram{2}\cdots\ydiagram{2}\]
These $i$ boxes correspond to the ordered basis $\set{t^{i-1}, t^{i-2}, \ldots, t, 1}$ of $k[t]/t^i$ so we can visualize the action of $t$ as moving each box one step to the left and annihilating the leftmost box.  More generally, given a linear operator $T\colon V \to V$ as above, the boxes in the Young diagram of $\jtype(T)$ correspond to basis elements in the basis of $V$ with respect to which $T$ is in Jordan normal form.  The action of $T$ moves each box one step to the left and the boxes in the leftmost column of the Young diagram correspond to the basis elements which span the kernel of $T$.

\begin{Def}
Let $j \in \mathbb N$.  The \emph{local $j$-rank} of $M$ is the function
\[\rank^j(-, M)\colon\PG \to \mathbb N_0\]
defined by $\rank^j(v, M) = \rank\theta^j_{v, M}$.
\end{Def}

It is useful to know that we can reconstruct $\jtype(v, M)$ if we know $\rank^j(v, M)$ for all $j$.  Note, for our linear operator $T$, that $\rank T^j$ is exactly the number of boxes that do not lie in the first $j$ columns of $\jtype(T)$.

\begin{Lem} \label{lemJrank}
Let $T\colon V \to V$ be a $p$-nilpotent linear operator.  Then $\jtype(T)$ is conjugate to the partition
\[(\rank T^0 - \rank T^1, \rank T^1 - \rank T^2, \ldots, \rank T^{p-1} - \rank T^p)\]
\end{Lem}
\begin{proof}
Follows immediately from the definition of a conjugate partition and the observation that $\rank T^{j-1} - \rank T^j$ is the number of boxes in the $j^\text{th}$ column of $\jtype(T)$.
\end{proof}

Thus the local $j$-ranks of a module $M$ encode the local Jordan type.  The local Jordan type in turn encodes the data of whether or not $M$ is projective.

\begin{Thm}[{\cite[7.6]{bfsSupportVarieties}}] \label{thmProj}
A module $M$ is projective if and only if its local Jordan type is the constant function $v \mapsto [p]^\frac{\dim M}{p}$.
\end{Thm}

We say that $M$ is \emph{projective at $v \in \PG$} if the local Jordan type at $v$ is $[p]^\frac{\dim M}{p}$, or more generally $M$ is projective on $U \subseteq \PG$ if it is projective at every point of $U$.  We say that $M$ has \emph{constant Jordan type} on a subset $U \subseteq \PG$ if the restriction of its local Jordan type to $U$ is a constant function, and \emph{constant $j$-rank} on $U$ if the restriction of its local $j$-rank to $U$ is a constant function.  We note that the previous theorem can be interpreted as a statement about the support of the module $M$.

\begin{Def}
The \emph{support}, $\PG_M$, of a module $M$ is the set of points $v \in \PG$ at which $M$ is not projective, i.e., at which $\jtype(v, M) \neq [p]^\frac{\dim M}{p}$.
\end{Def}

The support of $M$ is a closed subset of $\PG$~\cite[6.1]{bfsSupportVarieties}.  \Cref{thmProj} says that a module $M$ is projective (as a $G$-module) if and only if $\PG_M$ is the empty set or equivalently if and only if $M$ is projective on $\PG$ (i.e.\ projective at every point of $\PG$).

\section{The Associated Sheaves of a Module} \label{secAssSh}

Assume $M$ is a finite dimensional $G$-module.  In this section we consider the \emph{associated sheaves} of $M$, i.e., the kernel, image, and cokernel of powers of the global operator $\Theta_M\colon\widetilde M \to \widetilde M(p^{r - 1})$.  For $1 \leq i \leq p$ we also construct $\F_i(M)$ from these associated sheaves.  In addition to being independently motivated, the $\F_i(M)$ are important because they are isomorphic, up to a shift, to the quotients of filtrations of $\widetilde M$, cokernels of the global operator, and the sheaves $\Ho(M)$ and $\Ho[p-1](M)$ of \cref{secHo}.

In order to compose $\Theta_M$ with itself we must shift the degree of successive copies; hence we need a convention for which degrees we start and end at.  Given $j \in \mathbb N$ we define
\begin{align*}
\gker[j]{M} &= \ker\left[\Theta_M((j-1)p^{r-1})\circ\cdots\circ\Theta_M(p^{r-1})\circ\Theta_M\right], \\
\gim[j]{M} &= \im\left[\Theta_M(-p^{r-1})\circ\cdots\circ\Theta_M((1-j)p^{r-1})\circ\Theta_M(-jp^{r-1})\right], \\
\gcoker[j]{M} &= \coker\left[\Theta_M(-p^{r-1})\circ\cdots\circ\Theta_M((1-j)p^{r-1})\circ\Theta_M(-jp^{r-1})\right].
\end{align*}
Note that $\gker[j]{M}$ and $\gim[j]{M}$ are subsheafs of $\widetilde M$, and $\gcoker[j]{M}$ is a quotient of $\widetilde M$.  Also note that the canonical short exact sequence below now has a shift
\[0 \to \gker[j]{M} \to \widetilde M \to \gim[j]{M}(jp^{r-1}) \to 0.\]

Properties of the local Jordan type of $M$, or more specifically the $j$-rank of $M$, are related to the property that these sheaves are locally free.  We will detect this using the following lemma which is Exercise II.5.8 in Hartshorne~\cite{hartshorne} and a proof can be found in Friedlander and Pevtsova~\cite[4.11]{friedpevConstructions}.  We note that by \emph{specialization} of a sheaf at a point $x \in \PG$ we mean the functor $\mathscr F \mapsto k(x) \otimes_{\OPG[,x]} \mathscr F_x$ where one first takes the stalk at $x$ and then factors out the maximal ideal.  To ease the notation from here on out we will simply write $k(x) \otimes \mathscr F$ for the specialization of $\mathscr F$ at $x$.

\begin{Lem} \label{lemLocFree}
Let $X$ be a reduced Noetherian scheme and $\mathscr F$ a coherent sheaf over $X$.  Then the function
\[\phi(x) = \dim_{k(x)} k(x) \otimes \mathscr F\]
is upper semi-continuous, i.e., for any $n \in \mathbb Z$ the set $\set{x \in X \mid \phi(x) \geq n}$ is closed.  If $X$ is connected then $\mathscr F$ is locally free if and only if $\phi$ is constant.
\end{Lem}

The following theorem is a corrected version of a theorem by Friedlander and Pevtsova~\cite[4.13]{friedpevConstructions} which incorrectly equates constant rank with the image being locally free, as opposed to the cokernel.

\begin{Thm} \label{thm413}
Let $U \subseteq \mathbb P(G)$ be a connected open subscheme and $j$ a positive integer.  Given the following statements:
\begin{enumerate}
\item $M$ has constant $j$-rank on $U$, \label{thm413icjrk}
\item $M$ has constant $j$-rank on the closed points of $U$, \label{thm413icjrkcl}
\item $\gker[j]{M}|_U$ is locally free, \label{thm413ikerlf}
\item $\gim[j]{M}|_U$ is locally free, \label{thm413iimlf}
\item $\gcoker[j]{M}|_U$ is locally free, \label{thm413icokerlf}
\item $\ker\theta_{v, M}^j \simeq k(v) \otimes \gker[j]{M}$ for all $v \in U$ and $\gker[j]{M}|_U$ is locally free, \label{thm413ikercom}
\item $\im\theta_{v, M}^j \simeq k(v) \otimes \gim[j]{M}$ for all $v \in U$, \label{thm413iimloccom}
\item $\coker\theta_{v, M}^j \simeq k(v) \otimes \gcoker[j]{M}$ for all $v \in U$, \label{thm413icokerloccom}
\end{enumerate}
we have that \eqref*{thm413icokerloccom} always holds, \eqref*{thm413icjrk}, \eqref*{thm413icjrkcl}, \eqref*{thm413icokerlf}, \eqref*{thm413ikercom}, and \eqref*{thm413iimloccom} are all equivalent and imply \eqref*{thm413iimlf}, and \eqref*{thm413iimlf} implies \eqref*{thm413ikerlf}.
\end{Thm}
\begin{proof}
First note that specialization is a right exact functor and therefore commutes with taking cokernels, hence \eqref*{thm413icokerloccom} holds.  We also get from \cref{lemLocFree} that
\begin{equation}
\dim_{k(v)}\coker\theta^j_{v, M} = \dim M - \rank^j(v, M) \tag{$\ast$} \label{eqnCoker}
\end{equation}
is upper semi-continuous as a function of $v$; hence $\rank^j(-, M)$ is lower semi-continuous.  Now Hilberts Nullstellensatz says that the closed points contained in $U$ are dense and a continuous function that is constant on a dense set is constant.  Thus $\eqref*{thm413icjrk} \Leftrightarrow \eqref*{thm413icjrkcl}$.

Next recall that in any short exact sequence of sheaves of modules, if the middle and right sheaves are locally free then so is the left.  The canonical short exact sequences then give $\eqref*{thm413icokerlf} \Rightarrow \eqref*{thm413iimlf} \Rightarrow \eqref*{thm413ikerlf}$.  By \cref{lemLocFree}, \eqref*{thm413icokerlf} is equivalent to the statement that \eqref*{eqnCoker} is independent of the choice of $v \in U$, hence is equivalent to \eqref*{thm413icjrk}.

Assume \eqref*{thm413icokerlf} holds; then \eqref*{thm413ikerlf} and \eqref*{thm413iimlf} hold.  Specialization is exact when restricted to short exact sequences of locally free sheaves.  Specializing
\newcommand{\subU}[1]{\left.#1\right|_U}
\[\xymatrix@C=10pt{\subU{\gker[j]{M}} \ar@{^(->}[rr] && \subU{\widetilde M} \ar[rr]^-{\subU{\Theta_M^j}} \ar@{->>}[rd] && \subU{\widetilde M(jp^{r-1})} \ar@{->>}[rr] && \subU{\gcoker[j]{M}(jp^{r-1})} \\ &&& \makebox[3em][c]{\hss$\subU{\gim[j]{M}(jp^{r-1})}$\hss} \ar@{^(->}[ru]}\]
at $v \in U$ therefore gives
\[\xymatrix@C=10pt{k(v) \otimes \gker[j]{M} \ar@{^(->}[rr] && k(v) \otimes_k M \ar[rr]^-{\theta_{v, M}^j} \ar@{->>}[rd] && k(v) \otimes_k M \ar@{->>}[rr] && \coker\theta_{v, M}^j \\ &&& \makebox[2em][c]{\hss$k(v) \otimes \gim[j]{M}$\hss} \ar@{^(->}[ru]}\]
so \eqref*{thm413ikercom} and \eqref*{thm413iimloccom} hold.

Assume \eqref*{thm413iimloccom} holds.  Then by \cref{lemLocFree} the local $j$-rank of $M$ is both upper and lower semi-continuous on $U$.  As $U$ is connected it is therefore constant, hence \eqref*{thm413icjrk} holds.  Finally, assume \eqref*{thm413ikercom} holds.  Then \cref{lemLocFree} and
\[\dim_{k(v)}\ker\theta^j_{v, M} = \dim M - \rank^j(v, M)\]
immediately gives \eqref*{thm413icjrk}.
\end{proof}

\begin{Rmk}
For the Lie algebra $\slt$ we will see in \cref{propPslt} that $\PG[\slt] \simeq \mathbb P^1$ is a non-singular curve, so $\gker{M}$, a subsheaf of $\widetilde M$, is locally free even when $M$ does not have constant rank.  This shows that \eqref*{thm413ikerlf} does not imply \eqref*{thm413icjrk}.  Similarly one can check that the module $\Phi_{[0:1]}(4)$ provides a counterexample to $\eqref*{thm413iimlf} \Rightarrow \eqref*{thm413icjrk}$.  See Stark~\cite{starkHo2} for the definition of $\Phi_{[0:1]}(4)$ and additional $\slt$ sheaf computations.
\end{Rmk}

Consider a $p$-nilpotent linear operator $T\colon V \to V$.  Looking at the Young diagram of $\jtype(T)$ we see that $\ker T$ is a vector space whose dimension is the number of rows.
\begin{figure}[h]
\centering
\begin{picture}(240, 70)(20, 10)
\put(0, 50){\ydiagram[*(cyan)]{1,1,1,1,1} * [*(white)]{5,4,3,2,1}}
\put(120, 50){\ydiagram[*(cyan)]{3,2,1} * [*(white)]{5,4,3,2,1}}
\put(240, 50){\ydiagram[*(cyan)]{1,1,1} * [*(white)]{5,4,3,2,1}}
\put(15, 5){$\ker T$}
\put(135, 5){$\im T^2$}
\put(245, 5){$\ker T \cap \im T^2$}
\end{picture}
\end{figure}
As $T^i$ annihilates any row of length less or equal to $i$ we then get that the dimension of $\ker T \cap \im T^i$ gives the number of rows of length greater than $i$ in the partition.  Thus the dimension of the quotient space
\[\frac{\ker T \cap \im T^{i - 1}}{\ker T \cap \im T^i}.\]
gives the number of rows of length exactly $i$.  Motivated by this, Benson and Pevtsova~\cite{benPevtVectorBundles} make the following definition.

\begin{Def}
Given $1 \leq i \leq p$ we set
\[\F_i(M) = \frac{\gker{M} \cap \gim[i-1]{M}}{\gker{M} \cap \gim[i]{M}}.\]
\end{Def}

\begin{Lem} \label{lemFi}
Given $1 \leq i \leq p$ we have
\[\F_i(M)((i - 1)p^{r - 1}) \simeq \frac{\gker[i]{M} + \gim{M}}{\gker[i-1]{M} + \gim{M}}\]
\end{Lem}
\begin{proof}
The map
\[\gker[i]{M} \longrightarrow (\gker{M} \cap \gim[i-1]{M})((i - 1)p^{r - 1})\]
induced by $\Theta_M^{i-1}$ is surjective with kernel $\gker[i-1]{M}$.  Thus we have a commutative diagram
\[\xymatrix{(\gker{M} \cap \gim[i]{M})((i - 1)p^{r - 1}) \ar[r]^-{\text{incl}} & (\gker{M} \cap \gim[i-1]{M})((i - 1)p^{r - 1}) \\ \displaystyle \frac{\gker[i+1]{M}}{\gker[i]{M}}(-p^{r - 1}) \ar[u]^-{\Theta_M^i} \ar[r]_-{\Theta_M} & \displaystyle \frac{\gker[i]{M}}{\gker[i-1]{M}} \ar[u]_-{\Theta_M^{i-1}}}\]
whose vertical arrows are sheaf isomorphisms, so the sheaf $\F_i(M)((i - 1)p^{r - 1})$ in question is isomorphic to the cokernel of the bottom arrow.   The image of this bottom arrow is $(\gker[i]{M} \cap \gim{M} + \gker[i-1]{M})/\gker[i-1]{M};$ therefore, the second and third isomorphism theorems, together with the modular law, give
\[\F_i(M)((i - 1)p^{r - 1}) \simeq \frac{\gker[i]{M}}{\gker[i]{M} \cap \gim{M} + \gker[i-1]{M}} \simeq \frac{\gker[i]{M} + \gim{M}}{\gker[i-1]{M} + \gim{M}}\]
as desired.
\end{proof}

We will need two results from Benson and Pevtsova~\cite{benPevtVectorBundles}.  For the first we note that their proof, for the case of an elementary abelian group ($r = 1$), goes through in our more general setting once the appropriate shifts have been added.

\begin{Prop}[{\cite[2.2, 2.3]{benPevtVectorBundles}}] \label{propFilt}
The sheaf $\widetilde M$ has a filtration in which the filtered quotients are isomorphic to $\F_i(M)(jp^{r - 1})$ for $0 \leq j < i \leq p$.
\end{Prop}

As a corollary to the proof of this proposition we get the following.

\begin{Cor} \label{corFilt}
The sheaf $\gcoker[\ell]{M}$ has a filtration in which the filtered quotients are isomorphic to $\F_i(M)(jp^{r - 1})$ for $1 \leq i \leq p$ and $\max(0, i-\ell) \leq j < i$.
\end{Cor}
\begin{proof}
In the proof of the previous proposition Benson and Pevtsova~\cite[2.2]{benPevtVectorBundles} refine the kernel filtration
\[0 \subseteq \mathcal K \subseteq \mathcal K_2 \subseteq \cdots \subseteq \mathcal K_{p-1} \subseteq \widetilde M,\]
where $\mathcal K_i = \gker[i]{M}$, by the image filtration
\[0 \subseteq \mathcal I \subseteq \mathcal I_2 \subseteq \cdots \subseteq \mathcal I_{p-1} \subseteq \widetilde M,\]
where $\mathcal I_i = \gim[p-i]{M}$, to obtain a filtration with quotients
\[\frac{(\mathcal K_{j + 1} \cap \mathcal I_{i + 1}) + \mathcal K_j}{(\mathcal K_{j + 1} \cap \mathcal I_i) + \mathcal K_j} \simeq \frac{\mathcal K_{j + 1} \cap \mathcal I_{i + 1}}{(\mathcal K_{j + 1} \cap \mathcal I_i) + (\mathcal K_j \cap \mathcal I_{i + 1})}\]
which are isomorphic to $\F_{p - i + j}(M)(jp^{r - 1})$.  Note that this quotient is symmetric with respect to image vs.\ kernel so for the corollary we refine the image filtration by the kernel filtration and truncate below the term $\gim[\ell]{M}$ to obtain a filtration of $\gcoker[\ell]{M}$.
\end{proof}

The conditions determining which $\F_i(M)(jp^{r - 1})$ appear as filtered quotients become intuitively very clear if we accept the following maxim: \emph{We should think of $\F_i(M)(jp^{r - 1})$ as a sheafification of the $j^\text{th}$ boxes in rows of length $i$ in the Jordan type of $M$}.  Here we mean for the boxes to be counted starting at $0$ from left to right (see \cref{figFi}).
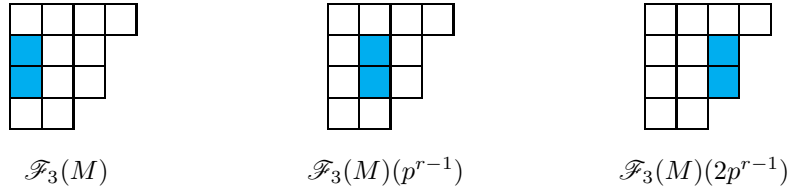
\begin{figure}[h]
\centering
\begin{picture}(240, 70)(20, 10)
\put(0, 50){\ydiagram[*(white)]{1} * [*(cyan)]{1,1,1} * [*(white)]{4,3,3,2}}
\put(120, 50){\ydiagram[*(white)]{2,1,1,1} * [*(cyan)]{2,2,2,1} * [*(white)]{4,3,3,2}}
\put(240, 50){\ydiagram[*(white)]{3,2,2} * [*(cyan)]{3,3,3} * [*(white)]{4,3,3,2}}
\put(5, 10){$\F_3(M)$}
\put(112, 10){$\F_3(M)(p^{r - 1})$}
\put(230, 10){$\F_3(M)(2p^{r - 1})$}
\end{picture}
\caption{Visualizing $\F_i(M)(jp^{r - 1})$} \label{figFi}
\end{figure}
In $\widetilde M$ we need, in rows of length $i$, boxes $0, 1, \ldots, i-1$.  On the other hand, in $\gcoker[\ell]{M}$ we only need the last $\ell$ boxes in a row.  If $i \leq \ell$ this is the entire row, otherwise it is only boxes $i - \ell, i - \ell + 1, \ldots, i-1$.

The second proposition needed from Benson and Pevtsova is \cref{propFOmega} below.  We note that Benson and Pevtsova's proof appears to implicitly assume that the module is constant Jordan type when they appeal to a ``block count'' (this relies on a previous proposition~\cite[2.1]{benPevtVectorBundles} which only applies to such modules).  We give an altered version of their proof which appeals only to a diagram chase and hence applies in the case when $M$ does not have constant Jordan type.  Recall that $\Omega(M)$, the \emph{Heller shift} of $M$, is defined to be the kernel of the map from the projective cover of $M$ to $M$.  

\begin{Prop}[{\cite[3.2]{benPevtVectorBundles}}] \label{propFOmega}
Let $M$ be a $G$-module and $1 \leq i < p$.  Then
\[\F_i(M) \simeq \F_{p-i}(\Omega M)((p-i)p^{r - 1}).\]
\end{Prop}
\begin{proof}
Consider the diagram
\[\xymatrix{ 0 \ar[r] & \widetilde{\Omega_M} \ar[r] \ar[d]^-{\Theta_{\Omega_M}} & \widetilde{P_M} \ar[r] \ar[d]^-{\Theta_{P_M}} & \widetilde M \ar[r] \ar[d]^-{\Theta_M} & 0 \\ 0 \ar[r] & \widetilde{\Omega(M)}(p^{r - 1}) \ar[r] \ar[d]^-{\Theta_{\Omega(M)}^{p-i-1}} & \widetilde{P_M}(p^{r - 1}) \ar[r] \ar[d]^-{\Theta_{P_M}^{p-i-1}} & \widetilde M(p^{r - 1}) \ar[r] \ar[d]^-{\Theta_M^{p-i-1}} & 0 \\ 0 \ar[r] & \widetilde{\Omega(M)}((p-i)p^{r - 1}) \ar[r] & \widetilde{P_M}((p-i)p^{r - 1}) \ar[r] & \widetilde M((p-i)p^{r - 1}) \ar[r] & 0.}\]
Using $\Theta_{P_M}^p = 0$ and a diagram chase in the style of the snake lemma we get a well defined map
\[\delta\colon\gker{M} \cap \gim[i-1]{M} \longrightarrow \frac{\gker{\Omega(M)} \cap \gim[p-i-1]{\Omega(M)}}{\gker{\Omega(M)} \cap \gim[p-i]{\Omega(M)}}((p-i)p^{r - 1}).\]
As $P_M$ is a projective module it has constant Jordan type; therefore, \cref{lemHoFree} gives $\gim[j]{P_M} = \gker[p-j]{P_M}$ for all $j$.  One then uses this in the diagram chase to show that $\delta$ is surjective and the kernel is exactly $\gker{M} \cap \gim[i]{M}$ so it induces the desired isomorphism.
\end{proof}

Finally we have the following proposition which is merely a local version of a global result of Benson and Pevtsova, which appears as \cref{corBPLocFree} below.  The proof follows their original proof save that we avoid their specialization argument and instead use a short exact sequence to deduce that the $\F_i(M)|_U$ are locally free.

\begin{Prop} \label{propFiFree}
Let $U \subseteq \PG$ be open.  The module $M$ has constant Jordan type $[p]^{a_p}[p-1]^{a_{p-1}}\cdots[1]^{a_1}$ on $U$ if and only if for each $i$ the sheaf $\F_i(M)|_U$ is locally free of rank $a_i$.
\end{Prop}
\begin{proof}
First assume $M$ has constant Jordan type.  There is a natural short exact sequence
\[0 \to \F_i(M)|_U \longrightarrow \frac{\gim[i-1]{M}|_U}{\gim[i]{M}|_U} \xrightarrow{\Theta_M} \frac{\gim[i]{M}}{\gim[i+1]{M}}(p^{r - 1})|_U \to 0;\]
therefore, to prove that each $\F_i(M)|_U$ is locally free we need only show that for each $i$ the quotient of $\gim[i-1]{M}|_U$ by $\gim[i]{M}|_U$ is locally free.  For that consider the short exact sequence
\[0 \to \frac{\gker[p-i]{M}|_U}{\gim[i]{M}|_U} \longrightarrow \frac{\gim[i-1]{M}|_U}{\gim[i]{M}|_U} \xrightarrow{\Theta_M} \gim[p-1]{M}(p^{r - 1})|_U \to 0.\]
By \cref{thm413} and \cref{lemHoFree} (whose proof does not rely on this result) the outer two sheaves are locally free therefore the middle sheaf is as well.

Next assume each $\F_i(M)|_U$, hence each $\F_i(M)(jp^{r - 1})|_U$, is locally free.  \cref{propFilt} then says that $\gcoker[\ell]{M}|_U$ has a filtration with locally free filtered quotients.  Inducting up this filtration we conclude that $\gcoker[\ell]{M}|_U$ is locally free.  By \cref{thm413} the module $M$ has constant $\ell$-rank on $U$.  This holds for all $\ell$ therefore $M$ has constant Jordan type on $U$.

Finally we must show that the ranks of the $\F_i(M)|_U$ give the Jordan type of $M$.  Let $v \in U$ be any point.  By \cref{thm413} the rank of the sheaf $\gim[i]{M}|_U$ is $\rank^i(v, M)$.  The short exact sequence
\[0 \to \gker{M}|_U \cap \gim[i]{M}|_U \longrightarrow \gim[i-1]{M}|_U \xrightarrow{\Theta_M} \gim[i]{M}(p^{r - 1})|_U \to 0\]
gives that $\gker{M}|_U \cap \gim[i]{M}|_U$ is locally free and its rank is $\rank^i(v, M) - \rank^{i-1}(v, M)$.  But by \cref{lemJrank} this is exactly the number of rows in $\jtype(v, M)$ of length greater or equal to $i$.  Finally, the defining short exact sequence
\[0 \to \gker{M}|_U \cap \gim[i]{M}|_U \longrightarrow \gker{M}|_U \cap \gim[i-1]{M}|_U \longrightarrow \F_i(M)|_U \to 0\]
gives that $\rank \F_i(M)|_U$ is exactly the number of rows of length $i$.
\end{proof}

\begin{Cor}[{\cite[7.4.12]{bensonElAb}, \cite[2.1]{benPevtVectorBundles}}] \label{corBPLocFree}
The module $M$ has constant Jordan type $[p]^{a_p}[p-1]^{a_{p-1}}\cdots[1]^{a_1}$ if and only if for each $i$ the sheaf $\F_i(M)$ is locally free of rank $a_i$.
\end{Cor}

\section{The Sheaf $\Ho(M)$} \label{secHo}

In this section we introduce the sheaf $\Ho(M)$ and are motivated to study the following question: Is it true that $M$ is projective if and only if $\Ho(M) = 0$?  We discuss and then improve upon previous results by Friedlander and Pevtsova, but we note that in general the statement is false.  We prove that if $\Ho(M) = 0$ then the support of $M$ is contained in the singular locus of $\PG$, hence the question has an affirmative answer in the smooth case.  Finally, we show that the full subcategory of modules $M$ such that $\Ho(M) = 0$ is a thick subcategory of the stable category.

To motivate the definition of $\Ho(M)$ consider a $p$-nilpotent linear operator $T$.  We have already noted that $\dim \ker T$ is the number of rows in $\jtype(T)$ and $\ker T \cap \im T^{p-1} = \im T^{p-1}$ so $\dim \im T^{p-1}$ is the number of rows of size $p$.  Thus the dimension of
\[H_T = \frac{\ker T}{\im T^{p - 1}}\]
gives the number of rows of size less than $p$ in $\jtype(T)$.
\begin{figure}[h]
\centering
\begin{picture}(240, 70)(20, 0)
\put(0, 40){\ydiagram[*(cyan)]{1,1,1,1} * [*(white)]{5,5,4,1}}
\put(120, 40){\ydiagram[*(cyan)]{1,1} * [*(white)]{5,5,4,1}}
\put(240, 40){\ydiagram[*(white)]{1,1} * [*(cyan)]{1,1,1,1} * [*(white)]{5,5,4,1}}
\put(15, 0){$\ker T$}
\put(135, 0){$\im T^4$}
\put(245, 0){$\frac{\ker T}{\im T^4}$}
\end{picture}
\end{figure}
By \cref{thmProj} a module $M$ is projective if and only if its Jordan type at any $v \in \mathbb P(G)$ has only rows of size $p$, i.e., if and only if $H_{\theta_{v, M}} = 0$ for all $v \in \PG$.  The global version of this space is the sheaf $\Ho(M)$.

\begin{Def}
Given $0 \leq i \leq p$ we define
\[\Ho[i](M) = \frac{\gker[i]{M}}{\gim[p-i]{M}}.\]
\end{Def}

\setcounter{Ques}{0}
\begin{Ques} \label{quesMain}
Is $M$ projective if and only if $\Ho(M) = 0$?
\end{Ques}

As noted in the introduction Friedlander and Pevtsova have given a partial answer to this question.  They show that a module $M$ is projective if and only if it has constant rank, constant $(p-1)$-rank, and $\Ho(M) = 0$~\cite[5.19]{friedpevConstructions}.  The forward direction of the next lemma shows that the assumption of constant $(p-1)$-rank is redundant.  The converse direction is a new proof of Proposition 5.16 in Friedlander and Pevtsova~\cite{friedpevConstructions}.

\begin{Lem} \label{lemHoFree}
Let $U \subseteq \PG$ be a connected open subscheme and $M$ a $G$-module which has constant $i$-rank on $U$.  Then $\Ho[i](M)|_U$ is locally free if and only if $M$ has constant $(p-i)$-rank on $U$.
\end{Lem}
\begin{proof}
We have $\widetilde M/\gker[i]{M} \simeq \gim[i]{M}(ip^{r-1})$ and hence a short exact sequence
\[0 \to \left.\Ho[i](M)\right|_U \to \left.\gcoker[p-i]{M}\right|_U \to \left.\gim[i]{M}(ip^{r-1})\right|_U \to 0.\]
By \cref{thm413} the sheaf $\gim[i]{M}|_U$, and hence the right sheaf in the sequence, is locally free.  We conclude that the left sheaf is locally free if and only if the middle sheaf is, hence if and only if $M$ has constant $(p-i)$-rank on $U$ (using \cref{thm413} again).
\end{proof}

\begin{Prop} \label{propH}
The sheaves $\Ho(M)$ and $\Ho[p-1](M)$ each have a $(p-1)$-step filtration whose quotients are the sheaves $\F_i(M)$ and $\F_i(M)((i - 1)p^{r - 1})$, respectively, for $i = 1, 2, \ldots, p-1$.
\end{Prop}
\begin{proof}
For $\Ho(M)$ the filtration is
\begin{align*}
\gim[p-1]{M} &= \gker{M} \cap \gim[p-1]{M} \\
&\subseteq \gker{M} \cap \gim[p-2]{M} \\
&\subseteq \cdots \\
&\subseteq \gker{M} \cap \gim{M} \\
&\subseteq \gker{M} \cap \gim[0]{M} = \gker{M}.
\end{align*}
For $\Ho[p-1](M)$ the filtration is
\begin{align*}
\gim{M} &= \gim{M} + \gker[0]{M} \\
&\subseteq \gim{M} + \gker{M} \\
&\subseteq \cdots \\
&\subseteq \gim{M} + \gker[p-2]{M} \\
&\subseteq \gim{M} + \gker[p-1]{M} = \gker[p-1]{M}
\end{align*}
and we use \cref{lemFi} to identify the quotients.
\end{proof}

\begin{Cor} \label{corH}
For each point $v \in \PG$ the stalk $\Ho(M)_v$ is zero if and only if $\Ho[p-1](M)_v$ is zero.
\end{Cor}
\begin{proof}
Localizing the filtrations from the previous proposition shows that $\Ho(M)_v$ and $\Ho[p-1](M)_v$ have filtrations with filtered quotients isomorphic to $\F_i(M)_v$ for $i = 1, 2, \ldots, p-1$, and a filtered object is zero if and only if each of its filtered quotients is.
\end{proof}

\begin{Thm} \label{thmHoProj}
Let $U \subseteq \PG$ be a connected open subscheme and assume that $\Ho(M)|_U = 0$.  Then the following conditions are equivalent.
\begin{enumerate}
\item $M$ is projective on $U$, \label{itemProj}
\item $M$ has constant rank on $U$,
\item $M$ has constant $(p-1)$-rank on $U$,
\item $\gcoker[p-1]{M}|_U$ is locally free,
\item $\gcoker{M}|_U$ is locally free,
\item $\gim[p-1]{M}|_U$ is locally free,
\item $\gim{M}|_U$ is locally free,
\item $\gker[p-1]{M}|_U$ is locally free,
\item $\gker{M}|_U$ is locally free. \label{itemKer}
\end{enumerate}
\end{Thm}
\begin{proof}
\Cref{corH} gives $\Ho[p-1]{M}|_U = 0$ as well therefore $\gker{M}|_U = \gim[p-1]{M}|_U$ and $\gker[p-1]{M}|_U = \gim{M}|_U$.  Now by \cref{thm413} condition (\ref*{itemProj}) implies that all of the remaining conditions hold and each of the above conditions implies (\ref*{itemKer}).  All that is left is to show that (\ref*{itemKer}) implies (\ref*{itemProj}).  For this we observe that \cref{propH} gives $\F_i(M)|_U = 0$ for $1 \leq i < p$ and by definition $\F_p(M)|_U = \gker{M}|_U$ so \cref{propFiFree} applies.
\end{proof}

We have already defined the support, $\PG_M$, of a $G$-module $M$ as the set of points at which $M$ is not projective and the support of a sheaf $\mathscr G$ is the set of points at which the stalk $\mathscr G_v$ is non-zero.  Thus \cref{quesMain} can be interpreted as asking if it is true that $\PG_M$ is empty if and only if $\supp\Ho(M)$ is empty.  Both are known to be closed subsets so, more generally, we ask the following.

\begin{Ques} \label{quesSupp}
Given a $G$-module $M$ do we have $\PG_M = \supp\Ho(M)$?
\end{Ques}

Friedlander and Pevtsova have given one inclusion and we can obtain the other on the regular locus of $\PG$ which we denote by $\reg\PG$.  We remind the reader that this is the set of points $v \in \PG$ such that the stalk, $\OPG[, v]$, of the structure sheaf is a regular local ring.

\begin{Prop}[{\cite[5.25]{friedpevConstructions}}]
Let $M$ be a $G$-module.  Then
\[\supp\Ho(M) \subseteq \PG_M.\]
\end{Prop}

\begin{Cor} \label{corHproj}
If $M$ is projective then $\Ho(M) = 0$.
\end{Cor}

\begin{Thm} \label{thmReg}
Let $M$ be a $G$-module.  Then
\[\PG_M \cap \reg\PG \subseteq \supp\Ho(M).\]
\end{Thm}
\begin{proof}
We prove the contrapositive.  Assume $v \in \reg\PG$ and the stalk $\Ho(M)_v$ is zero, then we will show that $M$ is projective in a neighborhood of $v$.

Consider the complex
\[\xymatrix@C=40pt{\cdots \ar[r]^-{(\Theta_M)_v} & \widetilde M_v \ar[r]^-{(\Theta_M^{p-1})_v} & \widetilde M_v \ar[r]^-{(\Theta_M)_v} & \widetilde M_v \ar[r]^-{(\Theta_M^{p-1})_v} & (\gker{M})_v \ar[r] & 0.}\]
As $\Ho(M)_v = 0$ we also have $\Ho[p-1](M)_v = 0$ by \cref{corH}; thus this is a resolution of $(\gker{M})_v$.  Note that it is a $2$-periodic resolution; the even syzygies are $(\gker{M})_v$.  As $v$ is regular $(\gker{M})_v$ has finite projective dimension so the syzygies of any projective resolution are eventually projective, hence free.  We may therefore take a connected neighborhood $U \subseteq \PG$ of $v$ such that $\gker{M}|_U$ is locally free.  By \cref{thmHoProj}, $M$ is projective on $U$.
\end{proof}

\begin{Cor} \label{corSupp}
If $\PG$ is regular then $\PG_M = \supp\Ho(M)$.
\end{Cor}

\begin{Cor} \label{corRegProj}
If $\PG$ is regular then $M$ is projective if and only if $\Ho(M) = 0$.
\end{Cor}

In general $\PG$ need not be regular; we will provide examples of when it is regular in the next section.  In the singular case we still have \cref{corHproj} so we will consider $\Ho$ as a functor out of the stable module category $\catd{G}{\underline{mod}}$.  Recall that the objects of $\catd{G}{\underline{mod}}$ are exactly the objects of $\catd{G}{mod}$.  For the hom-sets we define $\underline{\hom}_G(M, N)$ to be the quotient of $\hom_G(M, N)$ by the subspace of homomorphisms that factor through a projective.  We will consider the kernel of $\Ho$ as a subcategory of $\catd{G}{\underline{mod}}$.

\begin{Def}
Let $\K$ be the kernel of the functor
\[\Ho\colon\catd{G}{\underline{mod}} \to \cohPG,\]
that is, $\K$ is the full subcategory of $\catd{G}{\underline{mod}}$ consisting of those modules $M$ such that $\Ho(M) = 0$.
\end{Def}

The stable module category has a triangulated structure where the triangles are those candidate triangles that are isomorphic to
\[N \to M \to M/N \to \Omega^{-1}(N)\]
for some short exact sequence $0 \to N \to M \to M/N \to 0$ in $\catd{G}{mod}$.  We now show that $\K$ is a thick subcategory of $\catd{G}{\underline{mod}}$.  Recall that this means $\K$ is closed under isomorphisms, shifts, and taking summands and it satisfies the $2$-$3$ axiom: If 
\[X \to Y \to Z \to \Omega^{-1}(X)\]
is a triangle and $X, Y \in \K$ then also $Z \in \K$.

\begin{Thm} \label{thmThick}
$\K$ is a thick subcategory of $\catd{G}{\underline{mod}}$.
\end{Thm}
\begin{proof}
As $\Ho$ is functorial and clearly additive one has that $\K$ is closed under isomorphisms and taking summands.  For closure under the shift operations note that $\Ho(M) = 0$ if and only if $\F_i(M) = 0$ for all $1 \leq i < p$.  \cref{propFOmega} then gives $\Ho(\Omega M) = 0$ if and only if $\Ho(M) = 0$.  All that is left is the $2$-$3$ axiom.

Using closure under isomorphism we see that it suffices to show that $N \subseteq M$ and $\Ho(N) = \Ho(M) = 0$ imply $\Ho(M/N) = 0$.  As in the proof of \cref{thmReg} the hypotheses $\Ho(N) = \Ho(M) = 0$ imply that the complexes
\[\xymatrix@C=30pt{\cdots \ar[r]^-{\Theta_N} & \widetilde N((1 - p)p^{r - 1}) \ar[r]^-{\Theta_N^{p-1}} & \widetilde N \ar[r]^-{\Theta_N} & \widetilde N(p^{r - 1}) \ar[r]^-{\Theta_N^{p-1}} & \cdots}\]
and
\[\xymatrix@C=30pt{\cdots \ar[r]^-{\Theta_M} & \widetilde M((1 - p)p^{r - 1}) \ar[r]^-{\Theta_M^{p-1}} & \widetilde M \ar[r]^-{\Theta_M} & \widetilde M(p^{r - 1}) \ar[r]^-{\Theta_M^{p-1}} & \cdots}\]
are acyclic.  By naturality of global operators the inclusion $N \hookrightarrow M$ induces a map between these complexes and we find that
\[\xymatrix@C=30pt{\cdots \ar[r]^-{\Theta_{M/N}} & \widetilde{M/N}((1 - p)p^{r - 1}) \ar[r]^-{\Theta_{M/N}^{p-1}} & \widetilde{M/N} \ar[r]^-{\Theta_{M/N}} & \widetilde{M/N}(p^{r - 1}) \ar[r]^-{\Theta_{M/N}^{p-1}} & \cdots}\]
is the cokernel of a quasi-isomorphism and hence is acyclic, giving $M/N \in \K$.
\end{proof}

We now give an example where $\K$ has infinite representation type.  This shows that in general the answer to Questions \ref{quesMain} and \ref{quesSupp} is no.  One might then hope that $\K$ has ideal closure, as the tensor ideal thick subcategories of $\catd{G}{\underline{mod}}$ have been classified~\cite[6.3]{friedpevPiSupp}.  Unfortunately, the example below will also show that in general $\K$ is not a tensor ideal.  When $G$ is unipotent, i.e., when $k$ is the unique simple $G$-module, it is known that all thick subcategories are tensor ideals (c.f. Benson et al.~\cite[3.5]{bencarrickThickCat}) but in all known cases of unitary $G$ one has that $\PG$ is smooth and so \cref{corRegProj} already identifies $\K$.

Consider the Lie algebra $\slt[3]$ with $p \neq 2$.  The adjoint action of $\SL_3$ induces an action on $\PG[{\slt[3]}]$ with two orbits: The regular orbit, which is open and given as the set of lines through nilpotent matrices of rank $2$ (matrices of Jordan type $[3]$), and the sub-regular orbit, which is closed and given as the set of lines through nilpotent matrices of rank $1$ (matrices of Jordan type $[2][1]$).  The regular locus of $\PG[{\slt[3]}]$ is exactly the regular orbit~\cite[7.14]{jantzenNilpOrb}.

\begin{Ex} \label{exHCounter}
Assume $\mathrm{char} \ k = 3$ and recall that $\slt[3]$ can be considered as a height $1$ infinitesimal group scheme.  Let $M = k^3$ be the standard representation of $\slt[3]$.  Using Macaulay2~\cite{M2} one can compute that $\Ho(M) = 0$.  The support of $M$ is exactly the sub-regular orbit of $\PG[{\slt[3]}]$, so $M$ is not projective.  Moreover, the dimension of the support is well known to be the complexity of $M$.  As $M$ is indecomposable and the sub-regular orbit has dimension $4$, the Heller shifts $\Omega^n(M)$ are indecomposable of increasing dimensions.  These shifts therefore yield countably many non-isomorphic indecomposable modules contained in $\K$.

We can also see from this example that $\K$ is not closed under tensor products; in particular, it is not a tensor ideal thick subcategory.  This is also done using Macaulay2.  We first verify that $\Ho(M^\ast) = 0$ and then we compute the support of the sheaf $\Ho(M \otimes M^\ast)$ and observe that it is not empty.

The Macaulay2 code for these computations can be found in \cref{figMacCode}.  The variables {\tt theta}, {\tt thetaDual}, and {\tt thetaTens} contain the module homomorphisms corresponding to $\Theta_M$, $\Theta_{M^\ast}$, and $\Theta_{M \otimes M^\ast}$ respectively.

\begin{figure}[h]
\begin{minipage}{5in}
\begin{verbatim}
	i1 : R = ZZ/3[x1, x2, x3, y1, y2, y3, h7, h8]
	i2 : M = matrix{{h7, x1, x3},{y1, h8-h7, x2},{y3, y2, -h8}}
	i3 : kPG = R/(radical ideal M^3)
	i4 : theta = map(kPG^3, kPG^3, sub(M, kPG), Degree => 1)
	i5 : minimalPresentation(ker theta/image theta^2)
	o5 = 0
	i6 : thetaDual = -transpose theta
	i7 : minimalPresentation(ker thetaDual/image thetaDual^2)
	o7 = 0
	i8 : thetaTens = (theta ** id_(kPG^3)) + (id_(kPG^3) ** thetaDual)
	i9 : radical ann (ker thetaTens/image thetaTens^2)
	o9 = ideal (x3*y3 + h7*h8, x2*y3 + y1*h8, x1*y3 - y2*h7, y1*y2 +
	     -------------------------------------------------------------
	                                                     2
	     y3*h7 - y3*h8, x3*y2 + x1*h8, x2*y2 - h7*h8 + h8 , x3*y1 -
	     -------------------------------------------------------------
	     x2*h7, x1*x2 + x3*h7 - x3*h8)
\end{verbatim}
\end{minipage}
\caption{Macaulay2 code for \cref{exHCounter}} \label{figMacCode}
\end{figure}
\end{Ex}

\section{Regular support varieties} \label{secReg}

Given the results of the previous section it is useful to have examples of groups $G$ such that $\PG$ is regular.  We give several such examples in this section, begining with some obvious examples of Lie algebras $\mathfrak g$ such that $\PG[{\mathfrak g}]$ is regular.  Let $\mathfrak u_n$ and $\mathfrak b_n$ be the Lie algebras of $\U_n$ and $\B_n$ respectively.  Also, let $\mathfrak e_n$ be the elementary abelian Lie algebra of dimension $n$, i.e., $\mathfrak e_n$ is a vector space of dimension $n$ with trivial bracket and $p$-operation.
\begin{Prop} \label{propLiePG} \mbox{}
\begin{enumerate}
\item If $p \geq n$ then $\PG[{\mathfrak u_n}] \simeq \mathbb P^{\frac{1}{2}(n+1)(n-2)}$.
\item $\PG[{\mathfrak e_n}] \simeq \mathbb P^{n-1}$.
\item If $p \geq n$ then $\PG[{\mathfrak b_n}] \simeq \mathbb P^{\frac{1}{2}(n+1)(n-2)}$.
\end{enumerate}
\end{Prop}
\begin{proof}
For $(1)$ note that $p \geq n$ implies that every matrix in $\mathfrak u_n$ is $p$-nilpotent.  Thus $\PG[{\mathfrak u_n}]$ is the projective variety of lines in $\mathfrak u_n$ and
\[\dim_k\mathfrak u_n = \frac{1}{2}n(n-1) = \frac{1}{2}(n+1)(n-2) + 1.\]
A similar argument gives $(2)$.  For $(3)$ note that an upper triangular matrix is nilpotent if and only if it's strictly upper triangular, so for any $n$ the inclusion $\mathfrak u_n \to \mathfrak b_n$ induces an isomorphism of the $k$-points of their restricted nullcones.
\end{proof}

\begin{Prop} \label{propPslt}
The support variety of $\slt$ is $\PG[{\slt}] = \proj\frac{k[x, y, z]}{xy + z^2} \simeq \mathbb P^1$.
\end{Prop}
\begin{proof}
A matrix $\left[\begin{smallmatrix} z & x \\ y & -z \end{smallmatrix}\right]$ with entries in the field $k$ is nilpotent if and only if its square $(xy + z^2)\left[\begin{smallmatrix} 1 & 0 \\ 0 & 1 \end{smallmatrix}\right]$ is zero.  Thus the radical of the ideal generated by the relation $\left[\begin{smallmatrix} z & x \\ y & -z \end{smallmatrix}\right]^p = 0$ is $(xy + z^2)$ and immediately we get $\PG[{\slt}] = \proj\frac{k[x, y, z]}{xy + z^2}$.  The ring homomorphism
\begin{align*}
\frac{k[x, y, z]}{xy + z^2} & \to k[s, t] \\
(x, y, z) &\mapsto (s^2, -t^2, st).
\end{align*}
is easily seen to be injective and identifies $\frac{k[x, y, z]}{xy + z^2}$ with $k[s, t]^{[2]} = k[s^2, t^2, st]$.  The induced map $\iota\colon\proj k[s, t] \to \proj\frac{k[x, y, z]}{xy + z^2}$ is therefore an isomorphism of schemes.
\end{proof}

Recall that $\mathbb P(1, p, \ldots, p^{r-1})$ is the weighted projective space with weight vector $(1, p, \ldots, p^{r-1})$.  This means $\mathbb P(1, p, \ldots, p^{r-1}) = \proj k[w_1, \ldots, w_r]$ with grading given by $\deg w_i = p^{i-1}$.

\begin{Prop}
Let $n, r \in \mathbb N$.  Then
\[V(\B_{n(r)})_\red = V(\U_{n(r)})_\red\]
and
\[\PG[\B_{2(r)}] = \PG[\U_{2(r)}] = \mathbb P(1, p, \ldots, p^{r-1}).\]
\end{Prop}
\begin{proof}
As in the proof of \cref{propLiePG} observe that the restricted nullcones of $\mathfrak b_n$ and $\mathfrak u_n$ are identical therefore $V(\B_{n(r)})_\red = V(\U_{n(r)})_\red$ follows immediately from the description in \cref{exVlier}.  In particular we get $\PG[\B_{2(r)}] = \PG[\U_{2(r)}]$.  For $\PG[\U_{2(r)}] = \mathbb P(1, p, \ldots, p^{r-1})$, the map $\left[\begin{smallmatrix} 0 & a \\ 0 & 0 \end{smallmatrix}\right] \mapsto a$ yields an isomorphism $\U_2 \simeq \mathbb G_a$, hence $\PG[\U_{2(r)}] \simeq \PG[\mathbb G_{a(r)}]$, and the result follows from the description in \cref{exVGrt}.
\end{proof}

\begin{Cor}
The support varieties $\PG[\B_{2(r)}]$ and $\PG[\U_{2(r)}]$ are regular if and only if $r = 1, 2$.
\end{Cor}
\begin{proof}
This follows from the previous proposition once we know that the weighted projective space $\mathbb P(1, p, \ldots, p^{r-1})$ is regular if and only if $r \in \set{1, 2}$.  To see this we first observe that $\mathbb P(1) = \proj k[w_1] \simeq \spec k$ is regular and the $p^\text{th}$ truncation of $k[w_1, w_2]$ is $k[w_1^p, w_2] \simeq k[s, t]$ therefore $\mathbb P(1, p) \simeq \mathbb P^1$ is regular.  For $r > 2$ one checks directly that the non-vanishing locus of $w_3$ is the spectrum of the sub-ring
\[k[\varpi_0, \varpi_1, \ldots, \varpi_p, \sigma_4, \sigma_5, \ldots, \sigma_r] \subseteq k[w_1, \ldots, w_r]_{w_3},\]
where $\varpi_a = \frac{w_1^{ap}w_2^{p-a}}{w_3}$ and $\sigma_a = \frac{w_a}{w_3^{p^{a - 3}}}$.  We can think of this as a polynomial ring, with each $\varpi_a$ and $\sigma_a$ an indeterminant, modulo relations.  These relations are generated by $\varpi_a\varpi_b - \varpi_c\varpi_d$ when $a + b = c + d$; hence they are homogeneous and the spectrum of this ring has a singular point at the origin.  Thus $\mathbb P(1, p, \ldots, p^{r-1})$ has non-zero singular locus when $r > 2$.
\end{proof}

\section{DG-algebras and BGG correspondence} \label{secDG}

Let $k$ be a field of characteristic $2$ and define
\[\Lambda = \frac{k[y_1, \ldots, y_n]}{(y_1^2, \ldots, y_n^2)} \qquad \text{and} \qquad S = k[x_1, \ldots, x_n],\]
graded via $|x_i| = |y_i| = 1$ so that $\mathbb P^n = \proj S$.  We refer the reader to Appendix A of Okonek et al.~\cite{ossVectorBundles} for a detailed discussion of the classical Bern\u{s}te\u{\i}n, Gel'fand, Gel'fand (BGG) correspondence.  It is routine to check that the results therein hold in characteristic $2$ and give equivalences:
\[\catd{\Lambda}{\underline{grmod}} \equiv \Db{\catd{\Lambda}{grmod}}/\catd{\Lambda}{Perf} \equiv \Db{\catd{S}{grmod}}/\catd{S}{fdim} \equiv \Db{\coh(\mathbb P^n)}.\]
Here $\catd{\Lambda}{\underline{grmod}}$ is the graded stable module category of graded modules modulo projectives, $\Db{-}$ indicates the bounded derived category, $\catd{\Lambda}{Perf} \subseteq \Db{\catd{\Lambda}{grmod}}$ are the perfect complexes (i.e.\ those complexes that are isomorphic to a bounded complex of projective modules), and $\catd{S}{fdim} \subseteq \Db{\catd{S}{grmod}}$ are the complexes that are isomorphic to a bounded complex of finite dimensional modules.  In this section we give a generalization of these equivalences to DG-modules.  We indicate how the functor $\Ho$ factors through the resulting equivalence and use this to obtain further results.

Consider $S$ to be a DG-algebra with zero differential (we will use a subscript to indicate the grading of DG-$S$-modules and parenthesized superscripts to indicate the position in a chain complex of $\Lambda$-modules).  By $\D{S}$ we mean the derived category of DG-$S$-modules and let $\D{\catd{\Lambda}{Mod}}$ be the unbounded derived category of all $\Lambda$-modules (not necessarily finitely generated).  Given a chain complex $M$ of $\Lambda$-modules we define $G(M) = S \otimes_k M$.  We give $G(M)$ the structure of a DG-$S$-module via the grading $S_i \otimes M^{(j)} \subseteq G(M)_{i + j}$ and differential $\partial_{G(M)} = \id[S] \otimes \partial_M + \sum_\ell x_\ell \otimes y_\ell$.

\begin{Lem} \label{lemGdef}
$G\colon\D{\catd{\Lambda}{Mod}} \to \D{S}$ is a well defined functor.
\end{Lem}
\begin{proof}
We clearly have a well defined functor from the category of chain complexes of $\Lambda$-modules to the category of DG-$S$-modules, we need only show that it preserves quasi-isomorphisms.  A simple diagram chase shows that a map of chain complexes of $\Lambda$-modules, or a map of DG-$S$-modules, is a quasi-isomorphism if and only if both its kernel and cokernel are acyclic.  As $G$ is given by tensoring over a field it is exact therefore we reduce to showing that $G$ takes acyclic complexes to acyclic DG-$S$-modules.

Let $J = S \otimes_k \Lambda$ be the DG-$S$-module with grading $J_n = S_n \otimes \Lambda$ and differential $\partial_J = \sum_\ell x_\ell \otimes y_\ell$ and assume that $M$ is an acyclic complex of $\Lambda$-modules.  We observe that both $J$ and $G(M)$ have the structure of a complex of $\Lambda$-modules and one easily sees that $G(M)$ is the direct sum total complex of the double complex obtained by tensoring $J$ and $M$ over $\Lambda$.  As $J$ is positively graded and $M$ is exact the Acyclic Assembly Lemma~\cite[Lemma 2.7.3]{weibel} gives that $G(M)$ is also exact.
\end{proof}

Given an object $X$ in a triangulated category $\mathsf T$ we let $\thick[{\mathsf T}]{X}$ (or just $\thick{X}$ when there is no confusion) be the intersection of all thick subcategories that contain $X$, i.e., the smallest thick subcategory containing $X$.

\begin{Thm}
$G$ restricts to an equivalence $\Db{\catd{\Lambda}{mod}} \equiv \thick[\D{S}]{S}$.
\end{Thm}
\begin{proof}
Consider $I = S^\ast \otimes_k \Lambda$, where by $S^\ast$ we mean the graded dual defined by $(S^\ast)_n = \hom_k(S_{-n}, k)$.  This is a DG-$S$-module with grading $I_n = (S^\ast)_n \otimes_k \Lambda$ and differential $\partial_I = \sum_\ell x_\ell \otimes y_\ell$.  Note that $I$ also has the structure of a complex of $\Lambda$-modules, so if $M$ is also a complex of $\Lambda$-modules we may define $\hom_\Lambda^d(I, M)$ to be the set of all $\Lambda$-module homomorphisms of degree $d$ from $I$ to $M$ (not necessarily commuting with differentials).  Using the $S$-module structure of $I$ and the regular representation we get that $\hom_\Lambda^\bullet(I, M) = \bigoplus_d\hom_\Lambda^d(I, M)$ is a graded $S$-module; in fact it is a DG-$S$-module with differential $\phi \mapsto \partial_M\phi + \phi\partial_I$.

Let $\KInj{\Lambda}$ be the homotopy category of complexes of injective $\Lambda$-modules and, as in the proof of \cref{lemGdef}, let $J = S \otimes_k \Lambda$.  We have a well defined functor $\hom_\Lambda^\bullet(I, -)\colon\KInj{\Lambda} \to \D{S}$.  Benson et al.~\cite[Theorem 5.5]{bikLocCohSup} prove that the analogous functor $\hom_\Lambda^\bullet(J, -)\colon\KInj{\Lambda} \to \D{S}$ is an equivalence.  With only minor modifications the same proof shows that $\hom_\Lambda^\bullet(I, -)\colon\KInj{\Lambda} \to \D{S}$ is an equivalence as well.  The natural inclusion $\KInj{\Lambda} \to \D{\catd{\Lambda}{Mod}}$ restricts to an equivalence of categories $\KInj{\Lambda}^{\mathsf c} \equiv \Db{\catd{\Lambda}{mod}}$, where for any triangulated category $\mathsf T$ we let $\mathsf{T^c}$ denote the full subcategory of compact objects in $\mathsf T$ \cite[Proposition 2.3]{krauseStDerCat}.  The complex $I$ is a bounded above complex of projective modules, therefore $\hom_\Lambda^\bullet(I, -)$ preserves quasi-isomorphisms and consequently factors through this inclusion to give an equivalence $\hom_\Lambda^\bullet(I, -)\colon\Db{\catd{\Lambda}{mod}} \to \D{S}^{\mathsf c}$.
\[\xymatrix{\D{S} && \KInj{\Lambda} \ar[ll]_\equiv^{\hom_\Lambda^\bullet(I, -)} \ar[rr] && \D{\catd{\Lambda}{Mod}} \\ \D{S}^{\mathsf c} \ar@{^(->}[u] && \KInj{\Lambda}^{\mathsf c} \ar@{^(->}[u] \ar[ll]_\equiv \ar[rr]^\equiv && \Db{\catd{\Lambda}{mod}} \ar@/^25pt/[llll]^{\hom_\Lambda^\bullet(I, -)} \ar@{^(->}[u]}\]

All that's left is to identify $\D{S}^{\mathsf c}$ with $\thick{S}$ and show that $\hom_\Lambda^\bullet(I, -)$ is naturally isomorphic to $G$.  The first claim is well known, see for example Benson et al.~\cite[Theorem 2.2]{bikLocCohSup}.  For the second it suffices to show that $\hom_\Lambda^\bullet(I, M) \simeq G(M)$ when $M$ is a bounded complex of finitely generated $\Lambda$-modules.  For this we observe that the $d^\text{th}$ graded piece of $\hom_\Lambda^\bullet(I, M)$ is
\[\hom_\Lambda^d(S^\ast \otimes_k \Lambda, M) = \prod_i\hom_k((S^\ast)_i, M_{d + i}) = \prod_iS_{-i} \otimes_k M_{d + i}.\]
As $M$ is bounded, $M_{d + i}$ is nonzero for only finitely many $i$.  The product is therefore a direct sum and we get $\hom_\Lambda^d(I, M) \simeq (S \otimes_k M)_d = G(M)_d$.  One easily checks that this respects the differentials and is natural in $M$.
\end{proof}

There is an equivalence $\catd{\Lambda}{\underline{mod}} \equiv \Db{\catd{\Lambda}{mod}}/\catd{\Lambda}{Perf}$ defined by sending a $\Lambda$-module $M$ to the complex
\[\cdots \to 0 \to M \to 0 \to \cdots\]
with $M$ in the zeroth position and a zero in all other positions \cite[Theorem 2.1]{rickardStabEquiv}.  Using the natural inclusion $\catd{\Lambda}{\underline{grmod}} \hookrightarrow \catd{\Lambda}{\underline{mod}}$ we get the commutative diagram below.
\[\xymatrix{\catd{\Lambda}{\underline{grmod}} \ar@{^(->}[d] \ar[r]^-\equiv & \frac{\Db{\catd{\Lambda}{grmod}}}{\catd{\Lambda}{Perf}} \ar@{^(->}[d] \ar[r]^-\equiv & \frac{\Db{\catd{S}{grmod}}}{\catd{S}{fdim}} \ar[r]^\equiv & \Db{\coh(\mathbb P^n)} \\ \catd{\Lambda}{\underline{mod}} \ar[r]^-\equiv & \frac{\Db{\catd{\Lambda}{mod}}}{\catd{\Lambda}{Perf}}}\]

By $\Lambda \in \Db{\catd{\Lambda}{mod}}$ we mean its image under the equivalence above.  The perfect complexes are then exactly the elements of $\thick{\Lambda} \subseteq \Db{\catd{\Lambda}{mod}}$ \cite[Theorem 2.2]{bikLocCohSup} and $G(\Lambda) = J$, the DG-$S$-module from the proof of \cref{lemGdef}.  One can show that $J$ is quasi-isomorphic to the trivial DG-$S$-module $k$ (c.f.\ the map $\eta$ in Theorem 5.5 of Benson et al.~\cite{bikLocCohSup}) hence $G(\Lambda) \simeq k$.  This implies that $G$ restricts to an equivalence $\catd{\Lambda}{Perf} \equiv \thick{k}$ and consequently induces an equivalence between the quotients $\Db{\catd{\Lambda}{mod}}/\catd{\Lambda}{Perf} \equiv \thick{S}/\thick{k}$.  From the diagram above we see there is an induced inclusion $\Db{\catd{S}{grmod}}/\catd{S}{fdim} \hookrightarrow \thick{S}/\thick{k}$.  One can check that this is defined by mapping a complex $M$ to the total complex $\bigoplus_iM^{(i)}$ with grading given by declaring the elements of $M^{(i)}_j$ to be homogeneous of degree $i + j$.  The diagram is now as follows.
\[\xymatrix{\catd{\Lambda}{\underline{grmod}} \ar@{^(->}[d] \ar[r]^-\equiv & \frac{\Db{\catd{\Lambda}{grmod}}}{\catd{\Lambda}{Perf}} \ar[r]^-\equiv \ar@{^(->}[d] & \frac{\Db{\catd{S}{grmod}}}{\catd{S}{fdim}} \ar@{^(->}[d] \ar[r]^\equiv & \Db{\coh(\mathbb P^n)} \\ \catd{\Lambda}{\underline{mod}} \ar[r]^-\equiv & \frac{\Db{\catd{\Lambda}{mod}}}{\catd{\Lambda}{Perf}} \ar[r]^-\equiv & \frac{\thick{S}}{\thick{k}}}\]

If $M$ is a DG-$S$-module then the homology, $H(M)$, of $M$ is a graded module for the homology of $S$, which is just $S$.  A DG-$S$-module is contained in $\thick{S} \subseteq \D{S}$ (resp.\ $\thick{k}$) if and only if its homology is finitely generated (resp.\ finite dimensional) as an $S$-module~\cite[7.5]{abimHomPerf}.  Thus homology induces a functor $\thick{S}/\thick{k} \to \catd{S}{grmod}/\catd{S}{fdim} \equiv \coh(\mathbb P^n)$.  One can check that the functor $\Db{\coh(\mathbb P^n)} \to \coh(\mathbb P^n)$ defined by $\mathscr{G} \mapsto \bigoplus_iH^i(\mathscr{G})(-i)$ completes the diagram below.
\[\xymatrix{\catd{\Lambda}{\underline{grmod}} \ar@{^(->}[d] \ar[r]^-\equiv & \frac{\Db{\catd{\Lambda}{grmod}}}{\catd{\Lambda}{Perf}} \ar[r]^-\equiv \ar@{^(->}[d] & \frac{\Db{\catd{S}{grmod}}}{\catd{S}{fdim}} \ar[r]^\equiv \ar@{^(->}[d] & \Db{\coh(\mathbb P^n) \ar[d]} \\ \catd{\Lambda}{\underline{mod}} \ar[r]^-\equiv & \frac{\Db{\catd{\Lambda}{mod}}}{\catd{\Lambda}{Perf}} \ar[r]^-\equiv & \frac{\thick{S}}{\thick{k}} \ar[r]^-H & \coh(\mathbb P^n)}\]

\begin{Prop}
The functor $\catd{\Lambda}{\underline{mod}} \to \coh(\mathbb P^n)$ along the bottom row of the above diagram is identical to $\Ho\colon\catd{\mathbb G_{a(1)}^n}{\underline{mod}} \to \coh(\mathbb P^n)$.
\end{Prop}
\begin{proof}
We have $k\mathbb G_{a(1)}^r \simeq \Lambda$ so $\mathbb G_{a(1)}^r$-modules are equivalent to $\Lambda$-modules.  If $M$ is any such module then applying $G$ to the complex
\[\cdots \to 0 \to M \to 0 \to \cdots\]
yields the DG-$S$-module $S \otimes_k M$ with grading $|s \otimes m| = |s|$ and differential $\partial_{S \otimes_k M} = \sum_\ell x_\ell \otimes y_\ell$.  As an $S$-module this is exactly the graded module corresponding to the sheaf $\widetilde M$ and $\partial_{S \otimes_k M}$ is the homomorphism corresponding to $\Theta_M$.  Thus taking homology and then the associated sheaf yields $\Ho(M)$.
\end{proof}

Observe that this gives an alternate proof of \cref{corRegProj} in this case, for if $\Ho(M) = 0$ then $G(M)$ has finite dimensional homology and therefore is an element of $\thick{k}$, hence the equivalence gives $M \simeq 0$ in the stable module category.  We also deduce the following results about $\Ho$.  Note that in characteristic $2$ we have $\F_1 = \Ho$ so Benson and Pevtsova have proven a stronger version of \eqref*{thmItemSurj}~\cite{benPevtVectorBundles}.

\begin{Thm}
Let $k$ have characteristic $2$ and consider the group $G = \mathbb G_{a(1)}^n$.  The corresponding functor $\Ho$ is...
\begin{enumerate}
\item essentially surjective on objects, \label{thmItemSurj}
\item essentially surjective on maps, i.e., for all sheaf homomorphisms $\phi$ there is a module homomorphism $\psi$ such that $\Ho(\psi)$ equals $\phi$ up to isomorphisms in the domain and codomain, \label{thmItemSurjMap}
\item not a faithful functor. \label{thmItemNfaith}
\end{enumerate}
\end{Thm}
\begin{proof}
Both \eqref*{thmItemSurj} and \eqref*{thmItemSurjMap} follow immediately from the diagram and the fact that these properties obviously hold for the functor $\bigoplus_iH^i(-)(-i)\colon\Db{\coh(\mathbb P^n)} \to \coh(\mathbb P^n)$.  For \eqref*{thmItemNfaith} consider the trivial module $k$ and its first Heller shift $\Omega(k)$.  As $k$ is simple the stable homomorphisms from $k$ to $\Omega(k)$ are exactly the homomorphisms $\hom_G(k, \Omega(k)) \simeq k$.  But $\Ho(\Omega(k)) = \mathcal O_{\mathbb P^n}(-1)$ and $\Ho(k) = \mathcal O_{\mathbb P^n}$~\cite[3.5]{benPevtVectorBundles} and $\hom_{\mathcal O_{\mathbb P^n}}(\mathcal O_{\mathbb P^n}, \mathcal O_{\mathbb P^n}(-1)) = \Gamma(\mathbb P^n, \mathcal O_{\mathbb P^n}(-1)) = 0$.
\end{proof}

\bibliographystyle{plain}

\end{document}